\documentclass{amsart}
\usepackage{amsmath}
\usepackage{amssymb}
\usepackage{graphicx}
\usepackage{epsf}
\usepackage{subfigure}
\usepackage{amsfonts}
\usepackage{graphics}
\usepackage{graphicx}
\usepackage{enumitem}

\def\thalf{\text{\small$\frac{1}{2}$}}
\def\bequation{\begin{equation}}
\def\eequation{\end{equation}}
\def\baligned{\begin{aligned}}
\def\ealigned{\end{aligned}}

\def\argmin{\mathop{\rm arg\,min}}

\def\N{\mathcal{N}}
\def\D{\mathcal{D}}
\def\Gcal{\mathcal{G}}

\renewcommand{\Re}{\mathbb{R}}

\newcommand{\sd}{\partial}
\newcommand{\hsd}{\partial^\infty}
\newcommand{\csd}{\bar\partial}
\newcommand{\rsd}{\hat\partial}
\newcommand{\hcsd}{{\csd^\infty}}

\newcommand{\proj}[1]{\mathrm{proj}_{#1}\,}
\newcommand{\ncone}[2]{N_{#2}\left(#1\right)}

\newcommand{\bV}{{\bar V}}
\def\bdel{{\bar \del}}
\def\Del{{\Delta}}
\def\bv{{\bar v}}
\def\bg{{\bar g}}
\def\bt{{\bar t}}
\def\hht{{\hat t}}

\def\teps{{\tilde\eps}}
\def\tv{{\tilde v}}

\def\goesto{\rightarrow}

\def\epsopt{\eps_{opt}}
\def\nuopt{\nu_{opt}}
\def\redeps{\theta_\eps}
\def\rednu{\theta_\nu}

\newtheorem{thm}{Theorem}
\newtheorem{cor}{Corollary}
\newtheorem{lem}{Lemma}
\newtheorem{prop}{Proposition}
\newtheorem{defn}{Definition}
\newtheorem{rem}{Remark}

\newcommand{\epi}{\mathrm{epi}}
\newcommand{\graph}{\mathrm{graph}}
\newcommand{\norm}[1]{\left\Vert #1\right\Vert}
\newcommand{\dnorm}[1]{\left\Vert #1\right\Vert_*}
\newcommand{\set}[2]{\left\{#1\,\left|\, #2\right.\right\}}
\newcommand{\map}[3]{#1:\, #2\rightarrow #3}
\newcommand{\mmap}[3]{#1:\, #2\Rightarrow #3}
\newcommand{\dom}[1]{\mathrm{dom}\left(#1\right)}
\newcommand{\cl}{\mathrm{cl}\,}
\newcommand{\relin}{\mathrm{ri}\,}
\newcommand{\barrier}{\mathrm{bar}\,}

\newcommand{\uball}{I\!\!B}

\newcommand{\grad}{\nabla}
\newcommand{\conv}{{\mathrm{conv}\,}}
\newcommand{\dist}[2]{\mbox{dist}\left(#1\,\left|\, #2\right.\right)}

\newcommand{\cnc}[2]{{\bar{N}_{#2}(#1)}}

\newcommand{\gam}{\gamma}

\newcommand{\lam}{\lambda}

\newcommand{\eps}{\epsilon}
\newcommand{\sig}{\sigma}
\newcommand{\alf}{\alpha}
\newcommand{\del}{\delta}

\newcommand{\bR}{\mathbb{R}}
\newcommand{\bN}{\mathbb{N}}

\newcommand{\bB}{\mathbb{B}}

\newcommand{\bX}{\mathbb{X}}
\newcommand{\bY}{\mathbb{Y}}

\newcommand{\cQ}{\mathcal{Q}}

\newcommand{\cN}{\mathcal{N}}

\newcommand{\cV}{\mathcal{V}}

\newcommand{\cD}{\mathcal{D}}
\newcommand{\cH}{\mathcal{H}}

\newcommand{\bigtimes}{\mathop{\mathsf{X}}}

\newcommand{\tf}{{\tilde f}}

\newcommand{\bx}{{\bar x}}
\newcommand{\bd}{{\bar d}}
\newcommand{\hx}{{\hat x}}
\newcommand{\tx}{{\tilde x}}
\newcommand{\bz}{{\bar z}}
\newcommand{\hz}{{\hat z}}

\newcommand{\hk}{{\hat k}}
\newcommand{\bw}{{\bar w}}

\newcommand{\hK}{{\widehat{K}}}

\newcommand{\hJ}{{\widehat{J}}}

\newcommand{\balf}{{\bar \alf}}

\newcommand{\bnu}{{\bar \nu}}
\newcommand{\hbeta}{{\hat \beta}}

\newcommand{\hg}{{\hat g}}

\newcommand{\Rn}{\bR^n}

\newcommand{\cconv}{{\mbox{cl\! conv}\, }}

\newcommand{\bu}{{\bar u}}
\newcommand{\ip}[2]{\left\langle #1\, ,#2\right\rangle}

\newcommand{\intr}[1]{{\,\mathrm{int}\, #1}}
\newcommand{\R}{\mathbb{R}}

\newtheorem{example}{Example}

\begin{document}
\title[Non-Lipschitz Gradient Sampling]{Convergence of the Gradient Sampling Algorithm
on Directionally Lipschitz Functions}

\author{J.~V.~Burke}
\thanks{J.~V.~Burke, University of Washington, Seattle, WA.
\email{jvburke01@gmail.com}
Supported in part by the U.S. National Science Foundation grant DMS-1514559.
}
\author{Q.~Lin}
\thanks{Q.~Lin, Amazon Corp., 410 Terry Ave N., Seattle, WA.\\ 
\email{qiuyinglin5499@gmail.com}}

\begin{abstract}
The convergence theory for the gradient sampling algorithm 
is extended to directionally Lipschitz functions. 
Although directionally Lipschitz functions are not necessarily locally Lipschitz,
they are almost everywhere differentiable and 
well approximated by gradients and so are a natural candidate 
for the application of the gradient sampling algorithm. 
The main obstacle to this extension is the potential unboundedness
or emptiness of the Clarke subdifferential at points of interest.
The convergence analysis we present provides one path to
addressing these issues. In particular, we recover the usual convergence
theory when the function is locally Lipschitz. Moreover, if the
algorithm does not drive a certain measure of criticality to zero,
then the iterates must converge to a point at which either the Clarke
subdifferential is empty or the direction of steepest descent is
degenerate in the sense that it does lie in the interior
of the domain of the regular subderivative.

\smallskip

\noindent
This paper is dedicated to Terry Rockafellar on the occasion of
his 85th birthday. 
\end{abstract}

\maketitle

\section{Introduction}
The gradient sampling (GS) algorithm is designed to solve non-smooth optimization problems by using locally sampled gradients to
approximate the Clarke subdifferential and the associated direction of steepest descent. 
The objective 
is assumed to be continuously differentiable on an open set $\cD$ of full measure. 
Although the method was originally applied
to minimize non-Lipschitzian
nonsymmetric spectral functions 
\cite{BurkLewiOver02b,BurkLewiOver03,BurkLewiOver05},
the existing convergence theory
only applies to locally Lipschitz functions. 
The purpose of this note is to extend the convergence theory to 
{\it directionally Lipschitz functions} (see Definition \ref{def:DL}). 
Directionally Lipschitz functions
were introduced
by Rockafellar in \cite{RTR79b} and further developed in \cite{RTR80}. 
Loosely speaking, a function is directionally Lipschitz at a point $\bx$
if it is possible
to ``tilt'' its epigraph in such a way that the tilted set is the epigraph of a function
that is locally Lipschitz at $\bx$. 
A function can be directionally Lipschitz at a point but not locally Lipschitz or even continuous at that point. Some of the ideas for our approach 
appear in \cite{Lin09} and are motivated by the
results from \cite{BurkLewiOver02a,BurkLewiOver05,BCLOS20,Kiwi07}. 
In particular, our choice of directionally Lipschitz functions is inspired
by \cite[Corollary 6.1]{BurkLewiOver02a} (see Theorem \ref{thm:grad approx}) where it is shown that nearby gradients can be used to approximate their subdifferential. 
The primary difficulty in the non-Lipschitzian case is 
the potential
unboundedness or emptiness of the subdifferential.
Indeed, in this setting, it is not entirely clear 
what kind of convergence result can reasonably be expected. 

Both our choice of how the algorithm is stated and the consequent 
convergence theory closely parallels those proposed by Kiwiel in \cite{Kiwi07}
since his approach provides the most complete picture in the Lipschitzian case. 
A nice discussion of this approach as well as other recent advances and ongoing work is given in \cite{BCLOS20}. The paper proceeds as follows. 
Section \ref{sec:two} is broken into 4 parts: (1) notation and a review
of the subdifferential calculus especially the Clarke subdifferential and its relationship to the 
generalized (Mordukhovich or limiting) subdifferential, (2) pointedness of cones and its use in approximating the distance to a convex set, 
(3) the direction of steepest descent for nonsmooth functions, and 
(4) an introduction to
directionally Lipschitz functions.
In Section 3 we state the version of the gradient sampling algorithm to be examined and present our convergence results. 
We conclude in Section 4 with a few comments on on the algorithm and
its convergence.

\section{Preliminaries} \label{sec:two}

\subsection{Notation}
Our notation is based on that used in \cite{RW98}.
We work in the $n$-dimensional real Euclidean space $\Rn$
with the standard inner product $\ip{x}{y}$, with $\norm{\cdot}$
denoting the associated 2-norm whose closed unit ball is
$\bB:=\set{x\in\bX}{\norm{x}\le 1}$. 
Given $x\in\Rn$, define the open $\eps>0$ ball about $x$ as the set
$B_\eps(x):=\set{y}{\norm{x-y}<\eps}$.
For $C\subset\bX$, denote
the closure, interior, and convex hull of $C$ by $\cl C,\ \intr C$, and
$\conv C$, respectively. 
The distance to $C$ is defined by
$\dist{x}{C}:=\inf_{z\in C}\norm{x-z}$.
The set of natural numbers is denoted by 
$\bN:=\{1,2,\dots\}$.
The unit simplex in $\bR^{n+1}$ is given by $\Del_n:=\set{\lam\in \bR^{n+1}_+}{\lam_1+\dots+\lam_{n+1}=1}$.
Let $\bR_+$ the set of non-negative reals, and
$\bR_{++}$ the set of positive reals.

A set $K\subset\Rn$ is said to be a cone if $0\in K$ and $\lam x\in K$
for all $x\in K$ and $\lam\ge 0$. It is said to be a convex cone if it
is both a cone and a convex set. The cone $K\subset\Rn$ is said to be 
\emph{pointed} if for all $k\ge 2$ and $x^1,x^2,\dots,x^k\in K$ one has 
$x^1+x^2+\dots+x^k=0$ if and only if $x^i=0,\ i=1,\dots,k$.

The horizon cone and polar of $C\subset\Rn$ are given by
\[\begin{aligned}
C^\infty&:=\set{w}{\exists \{x^k\}\subset C,\ t_k\downarrow 0\text{ s.t. }t_kx^k\rightarrow w}\ \text{and}\\  
C^*&:=\set{v\in\bX^*}{\ip{v}{x}\le 1\ \forall, x\in C},
\end{aligned}
\]
respectively. The polar of a set is aways a closed convex set.
The convex indicator and support function for $C$
are give by 
\[
\del_C(x):=
\begin{cases}
0,&x\in C,\\ 
+\infty,&x\notin C
\end{cases}
\qquad\text{ and }\qquad
\del_C^*(v):=\sup_{x\in C}\ip{v}{x}\ ,
\]
respectively.

Given Euclidean spaces $\bX$ and $\bY$,
a mapping $S$ from $\bX$ to $\bY$ 
for which $S(x)$ is a subset of 
$\bY$ for every $x\in\bX$ (possibly empty) is called a multivalued mapping
and is denoted by $\mmap{S}{\bX}{\bY}$. The domain of $S$ is the set
$\dom{S}:=\set{x}{S(x)\ne\emptyset}$.  
Such a mapping $S$ is said to be outer semicontinuous (osc) if
\[
\set{v}{\exists\,(x^k,v^k)\rightarrow (x,v)\text{ with }v^k\in S(x^k)\, \forall\, k}
\subset S(x)\quad\forall\, x\in\dom{S}.
\]
The graph of $S$ is the set $\graph(S):=\set{(x,y)}{y\in S(x)}$ and the osc hull
of $S$ is the multivalued mapping $\mmap{\cl{S}}{\bX}{\bY}$ such that
$\graph(\cl{S})=\cl{\graph(S)}$. 

Let $\map{f}{\Rn}{\bar\bR:=\bR\cup\{+\infty\}}$ and set
\[\begin{aligned}
\dom{f}&:=\set{x}{f(x)<\infty}\\
\epi{f}&:=\set{(x,\mu)}{f(x)\le\mu}\subset\Rn\times\bR.
\end{aligned}\] 
Let $\bx\in\dom{f}$. The \emph{regular subdifferential} of $f$ at $\bx$
is given by 
$\rsd f(x)\!\!:=\!\!\set{v}{f(z)\ge f(x)+\ip{v}{z-x}+o(\norm{z-x})}$. 
This set is always closed and convex, but may be empty.
The \emph{subdifferential} of $f$ at $\bx$
is given by 
\[
\sd f(\bx)=\set{v}{\exists\, x^k\rightarrow \bx,\ 
v^k\rightarrow v\text{ s.t. }
v^k\in\rsd f(x^k)\ \forall\, k\in\bN},
\] 
and the 
\emph{horizon subdifferential} of $f$ at $\bx$
is given by 
\begin{equation}\label{eq:def hsd}
\hsd f (\bx):=\set{v}{
\begin{array}{c}
\exists\, x^k\rightarrow \bx,
\ t_k\downarrow 0,\ t_kv^k \rightarrow v,
\text{ s.t. }
\\
v^k\in\rsd f(x^k)\ \forall\, k\in\bN
\end{array}}.
\end{equation}
These sets are always closed, and if $f$ is lsc at $\bx$ then either $\sd f(\bx)\ne\emptyset$
or $\hsd f(\bx)$ contains at least one nonzero element
\cite[Corollary 8.10]{RW98}.
These subdifferentials are all mutivalued mappings with $\sd f$ and $\hsd f$ osc
along $f$-attentive sequences
by construction (an $f$-attentive sequence is any sequence $\{x^k\}\subset \dom{f}$
such that if $x^k\rightarrow \bx$ then $f(x^k)\rightarrow f(\bx)$).
Given $\eps>0$ define
 \[
 \csd_\eps f(\bx):=\set{v}{v\in \csd f(x)\text{ for some } x\in\bx+\eps\bB}
 \]
to be the $\eps$-approximate subgradients of $f$ at $\bx$.

Given a closed nonempty set $C\subset\Rn$ and a point $\bx\in C$, the 
regular normal cone to $C$ at $\bx$ is the set
\[
\widehat N_C(\bx):=\set{v}{\ip{v}{x-\bx}\le o(\norm{x-\bx})\text{ for }x\in C}.
\]
The osc hull of this multivalued mapping is called the normal
cone mapping and is denoted by $N_C(\bx)$. 
The \emph{Clarke normal cone} to $C$
a $x$ is given by $\cnc{x}{C}:=\cl\conv N_C(\bx)$. The cone of 
\emph{regular} tangents to $C$ at a point $x\in C$ 
where $C$ is locally closed is  
$\widehat T_C(x):=N_C(x)^*$ \cite[Theorem 6.28]{RW98}.

Given $\map{f}{\Rn}{\bar\bR:=\bR\cup\{+\infty\}}$
and $\bx\in\dom{f}$ at which $f$ is lsc, 
the \emph{Clarke subdifferential} of $f$ at $\bx$ is 
$\csd f(x)\!\!:=\!\!\set{v}{(v,-1)\in \bar{N}_{\epi{f}}(x,f(x))}$, and 
$\hcsd f(x)\!\!:=\!\!\set{v}{(v,0)\in \bar{N}_{\epi{f}}(x,f(x))}$ is
the
\emph{Clarke horizon subdifferential} of $f$ at $\bx$
\cite[Theorem 8.49]{RW98}.
The subdifferential and the Clarke subdifferential reduce to the usual
subdifferential in convex analysis when $f$ is convex.
Finally, the regular subderivative of $f$ at $x\in\dom{f}$, 
denoted $\map{\hat df(x)}{\Rn}{\bR\cup\{\pm\infty\}}$, 
at points where $f$ is lsc
is defined by the relation
$\epi(\hat df(x))=\widehat T_{\epi{f}}(x,f(x))$ \cite[Theorem 8.17]{RW98}.
The regular subderivative coincides with Clarke's directional derivative
when $f$ is locally Lipschitz \cite{Clar83}.
The following theorem establishes the relationships between the 
subdifferential and the Clarke subdifferential.

\begin{thm}[Subdifferential Relationships]
\cite[Theorem 8.49 and Exercise 8.23]{RW98}
\label{thm:basic sd}
Let $\map{f}{\Rn}{\Re}$ be locally lsc and finite-valued at $\bx\in\Rn$.
Then the following hold:
\begin{enumerate}
\item
$\sd f(x)$ and $\hsd f(x)$ are osc at $\bx$ with respect to $f-$attentive convergence, that is, with respect to sequences $\{x^k\}\subset \dom{f}$
such that $(x^k,f(x^k))\rightarrow (x,f(x))$.
\item 
$\csd f(\bx)$ is a closed convex set and $\hcsd f(\bx)$ is a closed convex cone.
\item 
$\hcsd f(\bx)=\csd f(\bx)^\infty$ when $\csd f(\bx)\ne\emptyset$, or equivalently,
$\sd f(\bx)\ne\emptyset$. 
\item
If the cone $\hsd f(\bx)$ is pointed 
(or equivalently, $\hcsd f(\bx)$ is pointed), then
\[
\csd f(\bx)= \conv \sd f(\bx) +\conv \hsd f(\bx)\ \mbox{ and }\ \hcsd f(\bx)=\conv \hsd f(\bx).
\]
\end{enumerate}
Moreover, if $\csd f(\bx)\ne\emptyset$ (equivalently,
$\sd f(\bx)\ne\emptyset$), then
\(
\hat df(x)=\del^*_{\csd f(x)}.
\)
\end{thm}

\subsection{Pointedness}
\noindent
We review pointedness and a few of its properties. 

\begin{lem}\label{lem:pointed cones}
Let $K$ be a non-empty closed cone in $\Rn$
and consider the following statements:
\begin{enumerate}
\item[(i)]
$K$ is pointed.
\item[(ii)] $K\cap(-K)=\{0\}$.
\item[(iii)] $\intr K^*\ne\emptyset$. 
\end{enumerate}
Statements (i) and (ii) are equivalent, and if $K$ is convex, both are equivalent to (iii). Moreover, in the convex case,
$z\in\intr K^*$ if and only if there exist $\eps>0$ such that
$\ip{z}{w}\le -\eps\norm{w}$ for all $w\in K$.
\end{lem}
\begin{proof}
The statements concerning (i)-(iii) follow from 
 \cite[Proposition 3.14, Exercise 6.22]{RW98}. 
 Therefore, we need only establish the 
the final statement of the lemma.
Let $z\in \intr K^*$ and 
$\eps>0$ be such that $z+\eps\bB\subset K^*$.
Then, for all $w\in K$ and $u\in\bB$, 
\(
0\ge\ip{z+\eps u}{w}=\ip{z}{w}+\eps\ip{u}{w}.
\)
Hence, \(0\ge \ip{z}{w}+\eps\sup_{u\in\bB}\ip{u}{w}=\ip{z}{w}+\eps\norm{w}\).

On the other hand, if there is a $z\in\Rn$ and $\eps>0$ is such that
$\ip{z}{w}\le -\eps\norm{w}$ for all $w\in K$, then, for all $u\in\bB$
and $w\in K$,
\(
\ip{z+\frac{\eps}{2}u}{w}\le -\eps\norm{w}+\frac{\eps}{2}\norm{w}
=-\frac{\eps}{2}\norm{w}
\)
so that $z\in \intr K^*$.
\end{proof}

We now connect the
pointedness of $C^\infty$
to projections
and the distance function for a non-empty closed convex set.
This result extends lemma \cite[Lemma 3.1]{Kiwi07} and
introduces a condition that is key to our analysis of the non-Lipschitzian
setting. 

\begin{lem}[Pointedness, and Projections]
\label{lem:dist bd and pointed hzn}
Let $C$ be a non-empty closed convex subset of $\Rn$ such that 
$C^\infty$ is pointed.
Let $z\notin C$ be such that
\bequation\label{eq:interior}
z-\proj{C}(z)\in \intr (C^\infty)^*\ .
\eequation 
Then, for all $\beta\in(0,1)$, there is a $\del>0$ such that if
$u,v\in C$ with 
$\norm{z-u}\le \dist{z}{C}+\del$, then $\ip{z-v}{z-u}>\beta \norm{z-u}^2$.  
In particular, if $z=0$, then $\ip{v}{u}>\beta\norm{u}^2$ whenever
$u,v\in C$ and $u$ satisfies $\norm{u}\le\dist{0}{C}+\del$.
\end{lem}

\begin{proof}
Let $\beta\in(0,1)$.
If the result were false, there exist a sequence $\{(u^k,v^k)\}\subset C\times C$ 
with $\norm{z-u^k}\le \dist{z}{C}+1/k$ such that
\begin{equation}\label{eq:hzn1} 
\ip{z-v^k}{z-u^k}\le \beta \norm{z-u^k}^2\quad\forall\, k. 
\end{equation}
Since $\{u^k\}$ is bounded we can assume with no loss of generality that 
$u^k\rightarrow \proj{C}(z)$.
The projection theorem tells us that 
\[ 
\ip{v-\proj{C}(z)}{z-\proj{C}(z)}\le 0\quad \forall\, v\in C,
\] 
or equivalently,
\begin{equation}\label{eq:hzn2} 
\dist{z}{C}^2\le \ip{z-v}{z-\proj{C}(z)}\quad\forall\, v\in C.
\end{equation}

If $\{v^i\}$ has a bounded subsequence, we can again assume with no loss in generality that 
$v^i\rightarrow \bv\in C$. Then, by \eqref{eq:hzn1}, $\ip{z-\bv}{z-\proj{C}(z)}\le \beta \dist{z}{C}^2$
which contradicts \eqref{eq:hzn2} since $\beta\in(0,1)$ and $\dist{z}{C}>0$. Hence,
the sequence $\{v^i\}$ is divergent. Consequently, we can assume,
with no loss in generality, that $v^i/\norm{v^i}\rightarrow \bv\in C^\infty$ with $\norm{\bv}=1$.
Dividing \eqref{eq:hzn1} by $\norm{v^i}$ and taking the limit yields $\ip{\bv}{z-\proj{C}(z)}\ge 0$.
But $\bv\in C^\infty$ and $z-\proj{C}(z)\in\intr (C^\infty)^*$, so, by 
Lemma \ref{lem:pointed cones},
there is an $\eps>0$ such that $\ip{\bv}{z-\proj{C}(z)}\le-\eps\norm{\bv}$. This contradiction
establishes the result.
\end{proof}

Condition \eqref{eq:interior} 
plays a central role
in our analysis of the GS algorithm.
The following lemma gives insight into this condition
by describing 
properties of the horizon cone $C^\infty$ and its
polar.

\begin{lem}[Normal, Barrier, and Horizon Cones]
\cite[Lemma 5]{BuD05}\label{lem:nbh cones}
Let $C$ be a non-empty closed convex set and define
$K:=\bigcup_{x\in C}N_C(x)$. Then 
$\relin{ \barrier C}\subset K\subset \barrier C$, and
\[
\cl K=\cl{\barrier C}=(C^\infty)^*
\ \mbox{ and }\
C^\infty=(\barrier C)^*,
\]
where $\barrier C:=\dom{\sig_C}$ is called
the barrier cone of $C$.
\end{lem}

Recall that it is always
the case that 
$z-\proj{C}(z)\in \ncone{\proj{C}(z)}{C}$, and so, by Lemma  
\ref{lem:nbh cones}, we have
\[
z-\proj{C}(z)\in \ncone{\proj{C}(z)}{C}\subset
\cl{\bigcup_{x\in C}N_C(x)}=(C^\infty)^*.
\]
In particular, if $C$ is bounded, then 
$(C^\infty)^*=\Rn$ so condition 
\eqref{eq:interior} is trivially satisfied.
Intuitively, 
the ``smaller''
the horizon cone of $C$ the more ``likely'' condition \eqref{eq:interior}
is satisfied.

\subsection{Steepest Descent Directions}
In the smooth setting the direction of steepest decent is given by 
the direction of unit length that minimizes the directional derivative.
By contrast, in the nonsmooth setting there are several notions
of directional derivative to choose from. 
From a numerical perspective, the
most useful permit a dual representation as the support function 
of an associated subdifferential which in turn yields a dual
representation of the direction of steepest descent via the Minimum Norm
Duality Theorem, e.g. see \cite[Theorem 2.8]{Burke83}. 

Since our analysis uses the Clarke subdifferential, our direction of
steepest descent is based on the regular subderivative
(see Theorem \ref{thm:basic sd}). That is,
the direction of steepest descent for $f$ at 
$x$ is given by
\bequation \label{eq:sd direction}
\bd_x:=\argmin_{\norm{x}\le 1}\hat df(x)(d)\ .
\eequation
The dual to this optimization problem is given by the Minimum Norm
Duality Theorem. 

\begin{thm} [Minimum Norm Duality Theorem]\cite{Nir61} 
\label{thm:MND}
Let $\bX$ be a normed linear space with norm $\norm{\cdot}$
and dual norm $\dnorm{\cdot}$, and let $\bB$ denote the
closed unit ball in $\bX$. Given a nonempty closed convex set 
$C\subset  \mathbb{X}^*$ and $\bz\in\bX^*$ with $\bz\notin C$, we have
\begin{equation}\label{eq:dist dual}
\text{dist}_*(\bz\mid C)=
\sup_{\norm{v}\le 1}[\ip{v}{\bz}-\del^*_C(v)],
\end{equation}
where $\del_C$ is the convex indicator of $C$ and  
$f^*$ denotes the convex conjugate of a function $f$. In particular,
if $\bz=0$, then
\[
\inf_{\norm{v}\le 1}\del^*_C(v)=-\text{dist}_*(0\mid C).
\]
\end{thm}

\noindent
The
Projection Theorem for convex sets tells us that for a nonempty closed convex
set $C$ and any $\bz\in\bX$ there is a unique vector $\hz\in C$ such that
$\dist{\bz}{C}=\norm{\bz-\hz}$, where $\hz$ is characterized by the condition that
$\bz-\hz\in N_C(\hz)$. The vector $\hz$ is called the projection of $\bz$ onto $C$ and is denoted by 
$\proj{C}(\bz)$. This implies that
$\bv:=\frac{(\bz-\proj{C}(\bz))}{\norm{\bz-\proj{C}(\bz)}}$ is the unique solution 
to the supremum problem in \eqref{eq:dist dual}.
By taking $C=\csd f(x)$, 
we obtain a dual interpretation for the direction of steepest descent.

\begin{cor}[Steepest Descent Duality] 
\cite[Theorem 2.8]{Burke83}
Let $\map{f}{\Rn}{\Re}$ and $x\in\dom{\sd f}$ be such that $f$ is lsc at $x$.
Then
\begin{equation}\label{eq:steepest descent}
\inf_{\norm{d}\le 1}\hat df(x)(d)=-\dist{0}{\csd f(x)},
\end{equation}
and the vector $\bd_x$ in \eqref{eq:sd direction} is given by
$\bd_x=- \proj{\csd f(x)}(\bz)/\norm{\proj{\csd f(x)}(\bz)}$.
\end{cor}
\begin{proof}
By Theorem \ref{thm:basic sd}, $\csd f(x)$ is a nonempty closed convex 
set with $\hat df(x)=\del^*_{\csd f(x)}$. The corollary follows 
by taking
$C=\csd f(x)$ and $\bz=0$ in Theorem \ref{thm:MND}.
\end{proof}

\subsection{Directionally Lipschitz Functions and Subdifferential Approximation}\label{sec:DL}

Rockafellar builds the notion of a directionally Lipschitzian function on that
of epi-Lipschitzian sets \cite{RTR79a}. He then establishes a useful characterization 
of directionally Lipschitzian functions through horizon subgradients 
\cite{RTR79b}.
We circumvent the epi-Lipschitzian construction and instead use the characterization given in \cite[Exercise 9.42]{RW98}
as our definition.

\begin{defn}[Directionally Lipschitzian Functions]\label{def:DL}
Suppose $\map{f}{\bR^n}{\bar\bR}$ 
is finite at $x\in\Rn$. We say $f$ is directionally
Lipschitz at $x$ if  
there is a unit vector $u$ and scalars $\eps>0$ and $M\in \R$
such that
\[
\frac{f(z+tv)-f(z)}{t}\le M\quad\forall\, v\in B_\eps(u),\, 
z\in B_\eps(x), f(z) \in B_\eps(f(x)),\, t\in(0,\eps].
\]
We say that $f$ is directionally
Lipschitz if it is directionally
Lipschitz at every point of $\Rn$.
\end{defn}

A simple characterization of directionally Lipschitz functions is 
obtained through the pointedness of the horizon cone of the 
subdifferential.

\begin{lem}\cite[Exercise 9.42(b)]{RW98}
A function $\map{f}{\bR^n}{\bar\bR}$ 
finite at $x\in\Rn$ is directionally
Lipschitz at $x$ if and only if 
$f$ is locally lsc at $x$ and the horizon 
subdifferential $\hsd f(x)$ is pointed.
\end{lem}

In particular, locally Lipschitz functions are directionally Lipschitz.
In \cite[Exercize 9.42(c)]{RW98}, Rockafellar and Wets show that
a function $\map{f}{\Rn}{\bar\R}$ that is finite and locally lsc at $\bx$
is directionally Lipschitz at $\bx$ if there is a convex cone 
$K\subset\Rn$ having nonempty interior such that $f$ is 
$K$-nonincreasing, i.e. $f(x+w)\le f(x)$ for all $x\in\Rn$ and $w\in K$.
In \cite[Theorem 6]{BBL03}, it is shown that if $\intr K\ne\emptyset$,
then $K$-monotone functions
($K$-nonincreasing or $K$-nondecreasing) are continuous and 
almost every where differentiable. These authors also
establish the following characterization of continuous 
directionally Lipschitz functions in terms of monotonicity.

\begin{prop}\cite[Proposition 8]{BBL03}
A continuous function $\map{f}{\bR^n}{\bar\bR}$ is
directionally Lipschitz at $x$ if an only if it is locally representable near $x$
as $f=g+l$ where $g$ is monotone with respect to a convex cone
with interior and $l$ is linear.
\end{prop}

\noindent
The pointedness of $\hsd f(x)$, or equivalently, $\hcsd f(x)$, is also
related to the continuity of
the regular subderivative $\hat df(x)$.

\begin{lem}[Continuity of $\hat d f(x)$]\label{lem:interior equiv}
Suppose $\map{f}{\bR^n}{\bar\bR}$ is finite 
at $x\in\Rn$ with $\sd f(\bx)\ne\emptyset$, then 
$\hat d f(x)$ is continuous on 
\[
\intr[(\hcsd f(\bx))^*]=\intr{\left[ \dom{\hat d f(\bx)(\cdot)}\right]}.
\]
\end{lem}
\begin{proof}
By Theorem \ref{thm:basic sd}, Lemma \ref{lem:nbh cones} and the closure properties of
convex sets, 
we have
\[\baligned
\intr[(\hcsd f(\bx))^*]&=\intr[(\csd f(\bx)^\infty)^*]
=\intr[\cl{\barrier \csd f(\bx)}]
\\ &=\intr[\barrier \csd f(\bx)]
=\intr[\dom{\del^*_{\csd f(\bx)}}]
\\ &=\intr{\left[ \dom{\hat d f(\bx)(\cdot)}\right]}.
\ealigned
\]
Since $\hat d f(\bx)$ is convex,
it is continuous on the interior of its domain.
\end{proof}
\noindent
In general, directionally Lipschitzian functions need not be locally Lipschitz or
even continuous at $\bx$. For example, for every $\eta\ge 0$, the function
\begin{equation}\label{eq:ex1}
f(x):=
\begin{cases}
x^{1/3}-\eta,&x\le 0,\\
x^{1/3}+\eta,&x>0,
\end{cases}
\end{equation}
is directionally Lipschitz at $\bx=0$ and continuous at $\bx=0$ if and only if
$\eta=0$.
Nonetheless, in \cite[Corollary 6.1]{BurkLewiOver02a} it is shown that when
$\hsd f(x)$ is pointed, then $\csd f(x)$
can be locally approximated  
by nearby gradients. We offer a slight improvement of this result that is useful to
our discussion. We begin with the following technical lemma.

\begin{lem}[Limits of Gradients]\label{lem:osc G}
Let $\map{f}{\Rn}{\bR}$ and $\bx\in\dom{f}$ be such that,
$\hsd f(\bx)$ is pointed, $f$ is continuous on an
open set $\cV$ containing $\bx$ and differentiable on an 
open set $\cQ\subset \cV$ of full measure in $\cV$. For each $x\in\cV$ and $\del>0$ such that 
$\bx+\del\bB\subset \cV$ set
\bequation\label{eq:G}
G_{\del}(x):=\cconv{\grad f((x+\del \uball)\cap \cQ)}.
\eequation
Let $x^k\rightarrow\bx$ and $\del_k\downarrow 0$ 
with $x^k+\del_k\bB\subset \cV$ for all $k\in\bN$. 
\begin{enumerate}
\item[(a)] If $w^k\rightarrow \bw$ with $w^k\in G_{\del_k}(x^k)$ for all $k\in\bN$,
then $\bw\in\csd f(\bx)$.
\item[(b)] If $v^k\rightarrow \bv$ with $v^k\in G_{\del_k}(x^k)^\infty$ for all $k\in\bN$,
then $\bv\in \hcsd f(\bx)$.
\end{enumerate}
\end{lem}
\begin{proof}
We only show (b) 
since the proof of (a)
follows the same
pattern but is significantly simpler.
By Carath\'eodory's Theorem, for each $k\in\bN$,
there exist sequences  
$\{(x^{kj1},\dots,x^{kj(n+1)})\,\mid\, j\in\bN\}\subset\bigtimes_{i=1}^{n+1}\Rn,\ 
\{\alf^{kj}\in\Del_n\,\mid\, j\in\bN\}$, 
and $\{t_{kj}\,\mid\, j\in\bN\}\subset\bR_+$ such that $t_{kj}\downarrow_j0$,
$x^{kji}\in (\bx+\del_k \uball)\cap \cQ\ ((j,i)\in\bN\times\{1,2,...,n+1\})$, and 
$t_{kj}\sum_{i=1}^{n+1}\alf^{kj}_i\nabla f(x^{kji})\overset{j}{\rightarrow}
v^k$.
Choose $\eps_k\downarrow 0$. For each $k\in\bN$, let $j_k\in\bN$ be such
that $t_{kj_k}<\eps_k$ and
$\norm{v^k-t_{kj_k}\sum_{i=1}^{n+1}\alf^{kj_k}_i\nabla f(x^{kj_ki})}\le \eps_k$.
For each $k\in\bN$, set 
$(\bx^{k1},\dots,\bx^{k(n+1)}):=(x^{kj_k1},\dots,x^{kj_k(n+1)})$,
$\balf^k:=\alf^{kj_k}$
and $\bt_k=t_{kj_k}$ so that 
$(\bx^{k1},\dots,\bx^{k(n+1)})\rightarrow(\bx,\dots,\bx)$, $\bt_k\downarrow 0$,
and
$\bt_{k}\sum_{i=1}^{n+1}\balf^{k}_i\nabla f(\bx^{ki})\rightarrow v$.
By compactness, we can assume that 
$\balf^k\rightarrow \balf\in\Del_n$.

Let us first suppose that the sequence 
$\{\bt_{k}(\nabla f(\bx^{k1},\dots,\nabla f(\bx^{k(n+1)}))\}$
is unbounded. If $\nu_k$ denotes the norm of the $k$th member
of this sequence, we can assume that $\nu_k\uparrow\infty$.
Then, with no loss in generality, there exists 
$(w^1,\dots,w^{(n+1)})$
such that 
\[
(\tilde t_{k1}\nabla f(\bx^{k1}),\dots,\tilde t_{k(n+1)}\nabla f(\bx^{k(n+1)}))
\rightarrow (w^1,\dots,w^{(n+1)})\ne(0,\dots,0),
\] 
where 
$\tilde t_{ki}:=(\bt_{k}\balf^{k}_i)/\nu_k$ for $i=1,\dots,(n+1)$ and $k\in\bN$.
Since $\tilde t_{ki}\downarrow 0$ and $\nabla f(\bx^{ki})\in\rsd f(\bx^{ki})$
for $i=1,\dots,n+1$ and $k\in\bN$, we have $w^i\in\hsd f(\bx)$ for 
$i=1,\dots,(n+1)$ (see \eqref{eq:def hsd}). But then 
$0=\sum_{i=1}w^i$ with $(w^1,\dots,w^{(n+1)})\ne(0,\dots,0)$ which
contradicts the pointedness of $\hsd f(\bx)$.
\\
Hence we can assume that the sequence 
$\{\bt_{k}(\nabla f(\bx^{k1}),\dots,\nabla f(\bx^{k(n+1)}))\}$
is bounded. Therefore, we may assume that there exist
$w^i\in\hsd f(\bx)$ such that $\bt_{k}\nabla f(\bx^{ki})\rightarrow w^i$
for $i=1,\dots,(n+1)$. Consequently, by Theorem \ref{thm:basic sd},
$v=\sum_{i=1}^{n+1}\balf_iw^i\in \conv \hsd f(\bx)=\hcsd f(\bx)$ proving the result.
\end{proof}

 \noindent
The outer semi-continuity of $\sd f$ and $\hsd f$ along $f$-attentive
 sequences \cite[Proposition 8.7]{RW98}
 implies that 
 pointedness is a local property and that
 the pointedness of  $\hsd f$ and $G_\del$
 are related. 
%

\begin{lem}[Pointedness of $\hsd f(x)$ is a local property]
\label{lem:loc pointedness}
Let $\map{f}{\Rn}{\Re}$ be continuous near $\bx\in\Rn$, and suppose
$\hsd f(\bx)$ is pointed. 
Let $\cQ$ be a full measure subset of a neighborhood of $\bx$ consisting of points where 
$f$ is differentiable, and let $G_\del(x)$ be as in \eqref{eq:G}. 
Then the following statements hold.
\begin{enumerate}
\item[(i)] 
There is an $\eps>0$ such that $\hsd f(x)$ is pointed on 
$(\bx+\eps\bB)\cap\dom{f}$.
\item[(ii)]
There is a $\bdel>0$ such that 
$G_\del(\bx)^\infty$ is pointed for all $\del\in (0,\bdel]$.
\item[(iii)]
There exist $\eps,\ \bdel>0$ such that both $\hsd f(x)$
and $G_\del(x)^\infty$ are pointed for all $x\in \bx+\eps\bB$ and $0<\del<\bdel$.
\end{enumerate}
\end{lem}
\begin{proof}
The statements (i)-(iii) are proved in essentially the manner. Therefore we only
prove (iii). If the result is false, then there exist sequences 
$\{x^k\}\subset\dom{\sd f}$ with $x^k\rightarrow \bx$ and $\del_k\downarrow 0$
such that either $\hsd f(x^k)$ is not pointed for all $k=1,2,\dots$ or
$G_{\del_k}(x^k)^\infty$ is not pointed for all $k=1,2,\dots$. Let us first suppose that
the cone $\hsd f(x^k)$ is not pointed for all $k=1,2,\dots$. Then there exist
$v^{k1},\, v^{k2}\in \hsd f(x^k)$ such that $v^{k1}+v^{k2}=0$ and
$\norm{v^{k1}}+\norm{v^{k2}}=1$ for all $k$. 
Compactness and the osc of $\hsd f$ at $\bx$ (Theorem \ref{thm:basic sd})
tells us that we can also assume there exist $\bv^1,\bv^2\in\hsd f(\bx)$ with
$(v^{k1}, v^{k2})\rightarrow (\bv^1,\bv^2)$, $\bv^1+\bv^2=0$ and 
$\norm{v^{1}}+\norm{v^{2}}=1$. But this contradicts the pointedness of 
$\hsd f(\bx)$. Next suppose that the cone
$G_{\del_k}(x^k)^\infty$ is not pointed for all $k=1,2,\dots$.
Again, there exist
$v^{k1},\, v^{k2}\in G_{\del_k}(x^k)^\infty$ such that $v^{k1}+v^{k2}=0$ and
$\norm{v^{k1}}+\norm{v^{k2}}=1$ for all $k\in\bN$. Compactness
tells us that 
we can assume there exist $\bv^1,\bv^2$
with $(v^{k1}, v^{k2})\rightarrow (\bv^1,\bv^2)$, $\bv^1+\bv^2=0$ and 
$\norm{v^{1}}+\norm{v^{2}}=1$. But Lemma \ref{lem:osc G}(b) tells us that
$v^1,v^2\in\hcsd f(\bx)$ which contradicts the pointedness of $\hcsd f(\bx)$.
\end{proof}

In the next lemma we establish a relationship between regular subgradients
and gradients at near by points. The lemma extracts a portion of the proof of 
\cite[Theorem 5.2]{BurkLewiOver02a} which we will use to extend  
\cite[Corollary 6.1]{BurkLewiOver02a}.

\begin{lem}[Gradients and Regular Subgradients]\label{lem:grad and rsg}
Let $\map{f}{\Rn}{\R}$ be continuous on $B_\del(\bx)$ for 
$\bx\in\Rn$ and $\del>0$, and assume that $\cQ$ is a full measure subset of 
$B_\del(\bx)$ consisting
of points where $f$ is differentiable. If either $f$ is absolutely continuous on line
segments in $B_\del(\bx)$ or $\cQ$ is open and $f$ is continuously differentiable on $\cQ$,
then
\(
\rsd f(\bx)\subset G_\del(\bx). 
\)
\end{lem}
\begin{proof}
If $f$ is absolutely continuous on line
segments in $B_\del(\bx)$, the result is the first statement established in the 
proof of \cite[Theorem 5.2]{BurkLewiOver02a}. 
If $\cQ$ is open and $f$ is continuously differentiable on $\cQ$, the result
requires only a very small change to this proof. 

If $y\notin G_\del(\bx)$, the separation theorem tells us that there exists
a non-zero vector $z$ and $k\in\R$ such that 
\[
\ip{y}{z}>k\ \text{but}\ \ip{\nabla f(x)}{z}\le k\ \forall\, x\in \cQ\cap(\bx+\del\bB).
\]
If $y\in\rsd f(\bx)$, then $f(x+tz)\ge f(\bx)+t\ip{y}{z}+o(t)$. Let $\bar t>0$ be such that
$t\ip{y}{z}+o(t)>kt$ for all $t\in(0,\bar t]$ so that
$f(\bx+t z)>f(\bx)+k\bar t$ for all
$ t\in (0,\bar t)$.
By continuity, given $t\in (0,\bar t)$, for all points $w$ sufficiently close to $\bx$,
$f(w+\bar t z)>f(w)+k\bar t$. Hence, we can choose $\bar w\in \cQ$ and
$\hat t\in (0,\bar t)$ so that 
\begin{equation}\label{eq:k ineq}
w+s z\in \cQ\cap(\bx+(\del/2)\bB)\ \forall\, s\in(0,\hat t]\  \ \text{with}\ \ 
f(w+\hat t z)>f(w)+k\hat t.
\end{equation}

Now consider the function $\map{g}{[0,\hat t]}{\R}$ defined by $g(s):=f(w+sz)$.
By construction $g$ is continuously differentiable on $(0,\hat t)$ with
$g'(s)=\ip{\nabla f(w+sz)}{z}\le k$. 
Therefore, by the Fundamental Theorem of Calculus,
$f(w+\hat t z)=g(\hat t)\le g(0)+k\hat t=f(w)+k\hat t$ which contradicts
\eqref{eq:k ineq}.
\end{proof}

\begin{thm}[Subdifferential Approximation]
\label{thm:grad approx}
Suppose that, close to $\bx\in\Rn$, the function $\map{f}{\Rn}{\Re}$ is continuous
and absolutely continuous on line segments, with $\hsd f(\bx)$ pointed.
If $\cQ$ is a full measure subset of a neighborhood of $\bx$ consisting of points where 
$f$ is differentiable, then 
\[
\csd f(\bx)=\bigcap_{\delta>0}G_{\del}(\bx)\quad\text{and}\quad
\csd^\infty f(\bx)=\bigcap_{\delta>0}G_{\del}(\bx)^\infty.
\]
Moreover, if $\cQ$ is open with $f$ continuously differentiable on $\cQ$, 
then the requirement that 
$f$ be absolutely continuous on line segments can be dropped.
%
\end{thm}
\begin{proof}
The statement of the theorem differs in two respects from the result
given in \cite[Corollary 6.1]{BurkLewiOver02a}.
First, the result in \cite{BurkLewiOver02a} makes no mention of the case when
$\cQ$ is open, and, second, there is no formula for the horizon cone
equivalence. The case when $\cQ$ is open follows from 
Lemma \ref{lem:grad and rsg} since the lemma tells us that the implication in
\cite[Theorem 5.2]{BurkLewiOver02a} follows from this hypothesis.
Consequently, \cite[Corollary 6.1]{BurkLewiOver02a} follows from this
hypothesis as well.

We now prove the horizon cone equivalence.  By Lemma \ref{lem:grad and rsg},
for all small $\del>0$, $\rsd f(x)\subset G_{\del/2}(x)\subset G_\del(\bx)$ 
for all $x\in B_{\del/2}(\bx)$. Hence $\hsd f(\bx)\subset G_\del(\bx)^\infty$,
and so, by Theorem \ref{thm:basic sd}(4), $\csd^\infty f(\bx)\subset G_\del(\bx)^\infty$ for all small $\del>0$. Consequently,
$\csd^\infty f(\bx)\subset\bigcap_{\delta>0}G_{\del}(\bx)^\infty$.
For the reverse inclusion let $v\in \bigcap_{\delta>0}G_{\del}(\bx)^\infty$.
Then there exist sequences $\del_k\downarrow 0$ and $v^k\rightarrow v$ 
such that $v^k\in G_{\del_k}(\bx)^\infty$ for all $k\in\bN$.
By Lemma \ref{lem:osc G}(b), $v\in\hcsd f(\bx)$ which proves the result.
\end{proof}

\noindent
The next lemma establishes a key property of approximate directions
of steepest descent for directionally Lipschitz functions and
extends the content of \cite[Lemma 3.1]{Kiwi07} to these functions.

\begin{lem}[Approximate Directions of Steepest Descent]
\label{lem:basic proj hyp1}
Let $\bx\in\Rn$ be such that $\map{f}{\Rn}{\Re}$ is differentiable on a full
measure subset $\cQ$ of an open convex neighborhood $\cN$ of $\bx$.
Further suppose that $f$ is either continuous and absolutely 
continuous along line segments in $\cN$ or that $\cQ$ is open. 
If 
\begin{equation}\label{eq:limits}
0\not\in\csd f(\bx),\  \emptyset\ne \csd f(\bx),\mbox{ and }
-\proj{\csd f(\bx)}(0)\in\intr(\hcsd f(\bx))^*,
\end{equation}
then, for all $\beta\in(0,1)$, there exists
$\del>0$ and $\eta>0$ such that 
$0\notin G_\eta(\bx)$ and, 
for every $u,v\in G_\eta(\bx)$ 
with $\norm{u}\le\dist{0}{G_\eta(\bx)}+\del$, we have
$\ip{v}{u}>\beta\norm{u}^2$. 
\end{lem}
\begin{proof}
%
Since $0\not\in\csd f(\bx)$, Theorem \ref{thm:grad approx} tells us that
there
is an $\bar\eta>0$ such that $0\notin G_\eta(\bx)$ for all $\eta\in(0,\bar\eta]$.
Theorem \ref{thm:grad approx} also tells us that
$\csd f(\bx)\subset G_\eta(\bx)$ for all $\eta\ge 0$. Therefore, 
$\dist{0}{G_\eta(\bx)}\le \dist{0}{\csd f(\bx)}<\infty$ for all $\eta\ge 0$.

We suppose the result is false and establish a contradiction.
Since the result is false, there exist $\hbeta\in(0,1)$ and sequences
$\{(u^i,v^i)\}$ in $\bR^{2n}$ and $\{(\eta_i,\del_i)\}$ in $\bR^2_+$
with $\eta_i\downarrow 0$ and $\del_i\downarrow 0$ such that,
for all $ i\in\bN$,
\begin{equation}\label{eq:G contro}
u^i,v^i\!\in\! G_{\eta_i}(\bx),\, \norm{u^i}\!\le\!\dist{0\!}{\!G_{\eta_i}(\bx)\!}+\del_i
\text{ and}\,\ip{v^i}{u^i}\!\le\!\hbeta\norm{u^i}^2. 
\end{equation}

Since $\{\norm{u^i}\}$ 
is bounded by $\dist{0}{\csd f(\bx)}+\del_0$, we may assume that $u^i\rightarrow \bu$, where 
$\bu\in\csd f(\bx)$ by Lemma \ref{lem:osc G}.
Theorem \ref{thm:grad approx} tells us that
$\csd f(\bx)\subset G_{\eta_i}(\bx)$ for all $i\in\bN$.
Hence, for all $i$ large,
\( 
\norm{u^i}\le\dist{0}{\csd f(\bx)}+\del_i.
\) 
%
Therefore, 
$\bu=\proj{\csd f(\bx)}(0)$ and so $-\bu\in \intr(\hcsd f(\bx))^*$
by \eqref{eq:limits}.

Next consider the sequence $\{v^i\}$. If this sequence is bounded, then, again,
Lemma \ref{lem:osc G} tells us that, with no loss in generality, there is a 
$\bv\in\csd f(\bx)$ such that $v^i\goesto \bv$. The projection theorem for convex sets
tells us that $\ip{\bv}{\bu}\ge\norm{\bu}^2$, but, by construction, 
$\ip{\bv}{\bu}\le\hbeta\norm{\bu}^2<\norm{\bu}^2$. 
This contradiction implies that the sequence $\{v^i\}$ is unbounded.
Therefore, with no loss in generality, $\{v^i\}$ is divergent.
By Carath\'eodory's Theorem, there exists $\lam^i\in\Del_n$ and $x^{ij}\in G_{\eta_i}(\bx)$
such that $v^i=\sum_{j=1}^{n+1}\lam_{ij}\nabla f(x^{ij})$ for all $i$. 
Since the sequence $\{v^i\}$ 
is divergent, the sequence defined by
$\hg_i:=(\lam_{i1}\nabla f(x^{i1}),\dots, \lam_{i(n+1)}\nabla f(x^{i(n+1)}))$ must also be divergent,  
and so, again with no loss in generality, 
there
is a $(\bg^1,\dots,\bg^{n+1})$ such that 
$\hg_i/\norm{\hg_i}\rightarrow (\bg^1,\dots,\bg^{n+1})\ne 0$, where we have taken 
$\norm{\hg_i}:=\max_{j=1,\dots,n+1}\lam_{ij}\norm{\nabla f(x^{ij})}$. 
Theorem \ref{thm:basic sd} tells us that
$\bg^j\in\hsd f(\bx),\, j=1,\dots,n+1$. 
Clearly, $\norm{v^i}\le\norm{\hg_i}$ for all $i=1,2,\dots$.
If $\{v^i/\norm{\hg_i}\}$ has a subsequence convergent to zero, 
then taking the limit
along this subsequence yields
$\sum_{j=1}^{n+1}\bg^j=0$
which contradicts the fact that $\hsd f(\bx)$ is pointed.
So we can assume that 
$v^i/\norm{\hg_i}=
\sum_{j=1}^{n+1}\lam^i_j\nabla f(x^{ij})/\norm{\hg_i}\rightarrow \tv
\in \hsd f(\bx)\setminus \{0\}$.
Theorem \ref{thm:basic sd} tells us that $\tv\in\hcsd f(\bx)=\csd f(\bx)^\infty$. 
But $-\bu\in\intr (\csd f(\bx)^\infty)^*$, so, by Lemma \ref{lem:pointed cones},
$\ip{\tv}{\bu}>0$ while $\ip{\tv}{\bu}\le 0$ by \eqref{eq:G contro}. This final contradiction establishes the result.
\end{proof}

The condition $-\proj{\csd f(\bx)}(0)\in\intr(\hcsd f(\bx))^*$ in 
\eqref{eq:limits}
plays an important role in our analysis.
Although examples where it fails to hold are easily generated,
such points are \emph{degenerate} in the sense that the direction of steepest descent for the regular subderivative does not lie
in the interior of its domain (see Lemma \ref{lem:interior equiv}).
Further discussion
of this issue is given in our concluding remarks.

\begin{example}
Let $\map{h}{\bR^2}{\bR}$ be given by
$h(x):=\ip{y}{x}+[\dist{x}{\bR^2_+}]^{1/2}$,
where $y:=(-1,\beta)^T$. Then
\(\csd h(0)=y+\bR^2_-,\ (\hcsd h(0))^*=\bR^2_+\) and
\[
-\proj{\csd f(\bx)}(0)=
\begin{cases}
(1,0)^T\not\in\intr(\hcsd f(0))^*&,\ \beta\ge 0,\\
(1,-\beta)\in\intr(\hcsd f(0))^*&,\ \beta< 0.
\end{cases}
\]
\end{example}

\section{The Gradient Sampling Algorithm}

Assume that
$\map{f}{\Rn}{\R}$ satisfies the following hypotheses:
\begin{center}\it
$\mathcal{H}$: $f$ is continuous on $\Rn$ and continuously differentiable \\ on an open
full measure set $\cD\subset\Rn$.
\end{center}

We use the form of the gradient sampling algorithm given in
\cite{BCLOS20} based on the version proposed by 
Kiwiel in \cite{Kiwi07}.

\noindent\rule{\textwidth}{1pt}

\noindent
 {\bf The GS Algorithm} (Gradient Sampling Algorithm) 

\noindent\rule{\textwidth}{1pt}
 
 \noindent
{\bf Initialization:} Let $x^0$ be a point at which $f$ is differentiable, 
choose termination tolerances $(\epsopt,\nuopt) \in [0,\infty) \times [0,\infty)$
and
the initial sampling radius $\epsilon_0 \in (\epsopt,\infty)$, initial stationarity target $\nu_0 \in [\nuopt,\infty)$, sample size $m \geq n+1$, line search parameters $(\beta,\gamma) \in (0,1) \times (0,1)$, and reduction factors $(\redeps,\rednu) \in (0,1] \times (0,1]$
    
\noindent
    {\bf For $k \in \N$ do}
\begin{enumerate}[label=(\roman*)]
      \item 
      \label{step.sample}
      Independently sample $\{x^{k,1},\dots,x^{k,m}\}$ uniformly from 
      $x^k+\epsilon_k\bB$. 
      \item 
      \label{step.terminate0}
      Terminate the algorithm if $\{x^{k,1},\dots,x^{k,m}\}\not\subset\cD$.
      \item 
      \label{step.qp}
      Compute $g^k$ as the solution of $\min_{g\in\Gcal^k} \thalf \|g\|^2$, where 
      \[\Gcal^k := \conv\{\nabla f(x^k),\nabla f(x^{k,1}),\dots,\nabla f(x^{k,m})\}.\] 
      \item 
      \label{step.terminate}
      \textbf{If} $\nabla f(x^k)=0$ or ($\|g^k\|_2 \leq \nuopt$ and 
      $\epsilon_k \leq \epsopt$), \textbf{then} terminate. 
      \item 
      \textbf{If} $\|g^k\|_2 \leq \nu_k$
      \item 
      \qquad \textbf{then} set $\nu_{k+1} \gets \rednu\nu_k$, 
      $\epsilon_{k+1} \gets \redeps\epsilon_k$, and $t_k \gets 0$
      \item 
      \label{step.armijo}
      \qquad \textbf{else}  set $\nu_{k+1} \gets \nu_k$, $\epsilon_{k+1} \gets \epsilon_k$, $d^k\gets -g^k/\norm{g^k}$, and 
      \begin{equation}\label{eq.armijo}
        \qquad\qquad t_k \gets \max\left\{ t \in \{1,\gamma,\gamma^2,\dots\} : 
        f(x^k + t d^k) < f(x^k) - \beta t \|g^k\| \right\}.
      \end{equation}
      \item 
      \label{step.check}
      \textbf{If} $f$ is differentiable at $x^k + t_k d^k $ 
      \item 
      \qquad \textbf{then} set $x^{k+1} \gets x^k + t_k d^k$
      \item 
      \label{step.update}
      \qquad \textbf{else} set $x^{k+1}$ randomly as any point where $f$ is 
      \item[] \qquad\qquad differentiable and such that 
      \begin{equation*}\label{eq.update}
      \begin{aligned}
        &f(x^{k+1}) < f(x^k) - \beta t_k \|g^k\|\ \ \text{and}\\ 
        &\ \|x^k + t_k d^k - x^{k+1}\|_2 \leq \min\{t_k,\epsilon_k\} 
        \end{aligned}\ .
      \end{equation*}
    \end{enumerate}
    {\bf End for}

\noindent\rule{\textwidth}{1pt}
\bigskip

\begin{rem}
As shown in \cite[Page 756]{BurkLewiOver05}, the line search  
\eqref{eq.armijo} in the algorithm is finitely terminating when 
$\nabla f(x^k)\ne 0$.
\end{rem}

\begin{rem}
In \cite[Section 4.1]{Kiwi07} it is observed that one can also take 
$d^k$ to be the un-normalized direction $-g^k$ when $f$ is Lipschitz.
However, the argument in \cite{Kiwi07} explicitly depends on $f$ being Lipschitz
continuous. In the non-Lipschitzian case, our proof of convergence requires 
the normalized direction in the statement of the GS algorithm given above.
\end{rem}

\subsection{Convergence}

We now introduce the key tools in analyzing the GS algorithm
introduced in \cite{BurkLewiOver05}: for $\eps,\del>0$ and $\bx,x\in\Rn$,
let
\bequation\label{eq:rho}
\rho_\eps(x):=\dist{0}{G_\eps(x)}
\eequation
and set
{ 
\[
\begin{aligned}
\D_\eps^m(x)&:=\prod_{1}^m((x+\eps\bB)\cap\D)\subset\Rn \qquad\text{and}\\
V_\eps(\bx,x,\del)\!\!&:=\!\!\set{\!(y^1,y^2,\dots,y^m)\!\!\in\! \D_\eps^m(x)\!\!}{
\!\text{dist}(0\!\mid \!\conv\!\{\nabla f(y^i)\}_{i=1}^m)\!\!\le\!\!\rho_\eps(\bx)\!\!+\!\!\del\!},
\end{aligned}
\]
}
where $m\ge n+1$ is as given in the statement of Algorithm I.

The next lemma shows that the convex hull of a collection of gradients
can be used to obtain directions of approximate steepest descent.

\begin{lem}\cite[Lemma 3.2(i)]{BurkLewiOver05}
\cite[Lemma 3.2(i)]{Kiwi07}\label{lem:loc uniformity}
Let $\eps > 0$ and $\bx\in \Rn$.
For all $\del>0$ there is a $\tau>0$ and a non-empty open set $\bV$ such that
$\bV\subset V_\eps(\bx,x,\del)$ for all $x\in B_\tau(\bx)$ with
$\dist{0}{\conv\{\nabla f(y^i)\}_{i=1}^m}\le \rho_\eps(\bx)+\del$ for all
$(y^1,\dots,y^m)\in \bV$.
\end{lem}

\begin{rem}
The statement of this lemma parallels the form 
given in \cite[Lemma 3.2(i)]{Kiwi07} rather than the form given in
\cite[Lemma 3.2(i)]{BurkLewiOver05}. 
Essentially the same proof is given in both papers and follows 
from the continuity of $\nabla f$ on $\cD$.
\end{rem}

We make use of the following mean value theorem to provide a lower bound
on the step sizes $t_k$ in step (vi) of the GS algorithm when 
$0\notin \sd f(x)$.
 
\begin{thm}[Approximate Mean Value Theorem]\cite[Theorem 3.4.7]{BoZ05}\label{thm:amvt}
Let $\map{\varphi}{\Rn}{\R}$ be lsc and assume that $r\in\R$ and $x,y\in\Rn$ are such that $x\ne y$,
$\varphi(x)<+\infty$, and $r<\varphi(y)-\varphi(x)$. 
Then there is a $\hx\in [x,y)$ such that for all $\eps>0$
there exists $(\tx,\varphi(\tx))\in B_\eps((\hx,\varphi(\hx)))$ and 
$\tv\in\rsd \varphi(\tx)$ for which
\[
\ip{\tv}{\hx-\tx}>-\eps,\ \ \ip{\tv}{y-x}>r,\ \mbox{ and }\ \varphi(\tx)\le \varphi(x) +|r|+\eps.
\]
\end{thm}

\begin{lem}[Stepsize Bound]\label{lem:stepsize}
Let $\map{f}{\Rn}{\R}$ be such that $\cH$ holds.
Let $\beta,\gam\in (0,1)$ be given, and let $\bx\in\Rn$ be such that
all three conditions in \eqref{eq:limits} hold.
Then there exist 
$\eta>0$ and $\del>0$ so that 
the consequences of Lemma \ref{lem:basic proj hyp1} hold. 
Moreover,
given $\eps>0$, we can choose $\tau\in(0,\eps/3)$ so that the consequences of
Lemma \ref{lem:loc uniformity} hold for this $\del$. 
That is, there exists 
a non-empty open set $\bV$ such that
$\bV\subset V_\eps(\bx,x,\del)$ for all $x\in B_\tau(\bx)$ with
$\dist{0}{\conv\{\nabla f(y^i)\}_{i=1}^m}\le \rho_\eps(\bx)+\del$ for all
$(y^1,\dots,y^m)\in \bV$. 
Then, for all $x\in B_\tau(\bx)$ 
and $(x^1,\dots,x^m)\in \bV$,
\begin{equation}\label{eq:stepsize1}
\bt\!:=\!\min\{1,\gam\eps/3\}\!\le\! \hht\!:=\!
\max\set{t}{\begin{aligned}&f(x+td)\!<\!f(x)-\beta t\norm{g}\\ &t\in\{1,\gam,\gam^2,\dots\}\end{aligned}},
\end{equation}
where $g:=\argmin\set{\norm{v}}{v\in\conv\{\nabla f(x^i)\}_{i=1}^m\}}$ and
$d:=-g/\norm{g}$.
\end{lem}
\begin{proof}
Since the hypotheses of Lemmas \ref{lem:basic proj hyp1} and
\ref{lem:loc uniformity} are satisfied, the parameters $\eta,\, \del$ and $\tau$ can be
chosen as required.
Set $\widehat G:=\conv\!\{\nabla f(x^i)\}_{i=1}^m$. Since 
$(x^1,\dots,x^m)\!\in \!\bV\!\subset\! V_\eps(\bx,\bx,\del)$, 
Lemma \ref{lem:loc uniformity} tells us
that $\dist{0}{\widehat G}\le \rho_\eps(\bx)+\del$ and 
$\widehat G\subset G_\eps(\bx)$.
Hence, $g\in G_\eps(\bx)$ and $\norm{g}\le\rho_\eps(\bx)+\del$. 
Consequently, Lemma \ref{lem:basic proj hyp1} tells us that
\begin{equation}\label{eq:stepsize2}
\ip{v}{g}>\beta\norm{g}^2\quad\forall\ v\in G_\eps(\bx). 
\end{equation}
Assume to the contrary that the inequality \eqref{eq:stepsize1} is false.
Then $\hht<1$ and so
\[
-\beta \gam^{-1}\hht\norm{g}\le f(x+\gam^{-1}\hht d)-f(x).
\]
By taking $f=\varphi$, $x=x,\ y=x+\gam^{-1}\hht d$ and
$r=-\beta \gam^{-1}\hht\norm{g}$ in Theorem \ref{thm:amvt},
there exists $\hx\in[x,\, x+\gam^{-1}\hht d)$ such that for all $\teps>0$ there exists
$(\tx,f(\tx))\in B_\eps(\hx,f(\hx))$ and $\tv\in\rsd f(\tx)$ such that
\[
-\beta \gam^{-1}\hht\norm{g}< \gam^{-1}\hht\ip{\tv}{d},
\]
or equivalently,
\[
\ip{\tv}{g}< \beta \norm{g}^2.
\]
So $\tv\notin G_\eps(\bx)$ by \eqref{eq:stepsize2}.
Assume that we have chosen $\teps\in(0,\eps/3)$.
Since the inequality \eqref{eq:stepsize1} is false, 
$\hht<\gam\eps/3$ or equivalently,
$\gam^{-1}\hht \norm{d}<\eps/3$.
Consequently, $\tx\in B_\eps(\bx)$ and $\sd f(\tx)\subset G_\eps(\bx)$.  
Therefore, $\tv\in\rsd f(\tx)\subset\sd f(\tx)\subset G_\eps(\bx)$. This contradiction establishes 
the result.
\end{proof}

The main convergence result for the GS Algorithm now follows. Our proof
is inspired by Kiwiel's proof of \cite[Theorem 3.3]{Kiwi07}.

\begin{thm}[Convergence: $0=\nu_{opt}=\eps_{opt}$]\label{thm:convergence}
Suppose $\map{f}{\Rn}{\bR}$ satisfies $\cH$. 
Let $\{x^k\}$ be a sequence generated by the GS Algorithm with
$\nu_0,\, \eps_0\in\bR_{++}$, $\theta_\eps,\,\theta_\nu\in(0,1)$, and $\eps_{opt}=\nu_{opt}=0$.
With probability 1 the algorithm does not terminate 
in line \ref{step.terminate0} 
and 
one of the following must occur: 
\begin{enumerate}[label=(\alph*)]
\item 
There is a $k_0\in\bN$ such that $\nabla f(x^{k_0})=0$ and 
the algorithm terminates.
\item 
$f(x^k)\downarrow -\infty$. 
\item \label{limit1}
$0<\bar\nu:=\inf_k\nu_k$ and the sequence converges 
to some $\bx\in\Rn$ for which at least one of the three
conditions in \eqref{eq:limits} must be violated, that is, either
\begin{equation}\label{eq:Limits}
\emptyset=\csd f(\bx),\quad 0\in\csd f(\bx),\quad \text{ or } 
-\proj{\csd f(\bx)}(0)\notin\intr(\hcsd f(\bx))^*. 
\end{equation}
\item \label{limit2}
$\nu_k\downarrow 0$ and every cluster point $\bx$ of $\{x^k\}$
(if one exists) 
satisfies $0\in\csd f(\bx)$.
\end{enumerate}
Moreover, if $f$ is locally Lipschitz, then outcome (c) cannot occur.
\end{thm}
\begin{proof} If $f$ is locally Lipschitz, then the result follows
from \cite[Theorem 3.3]{Kiwi07}. 
The assumptions on the function $f$ imply that, with probability 1, the 
algorithm does not terminate in line \ref{step.terminate0}
of the GS Algorithm, so we assume 
$\{x^{k,1},\dots,x^{k,m}\}\subset\cD$ for all $k$.
We also assume that neither (a) nor (b) occur and show that either 
(c) or (d) must occur. 
Let $J\subset\bN$ be those iterations for which $x^k\ne x^{k+1}$.
Observe that if $J$ is finite with
maximum value $k_0$, then $\nabla f(x^{k_0})=0$, 
hence, $J$ is infinite.
Since
\[
f(x^{k+1})\le f(x^k)-\beta t_k\norm{g^k}\quad \forall\, k\in \bN,
\]
the sequence $\{f(x^k)\}$ is non-increasing and bounded below, and
so has has a limit $\tf$. Summing this inequality over $k$ and taking the 
limit tells us that
\begin{equation}\label{eq:finite sum}
\beta\sum_{k=1}^\infty \norm{x^{k+1}-x^k}\norm{g^k}\le
\beta\sum_{k=1}^\infty t_k\norm{g^k}\le f(x^0)-\tf <\infty,
\end{equation}
where the first inequality follows from lines (viii)-(x) of the GS algorithm.
In particular, $\norm{x^{k+1}-x^k}\norm{g^k}{\rightarrow} 0$. 
We decompose this fact into
two mutually exclusive possibilities: either 
$0<\bar\nu:=\inf_k\nu_k$ or 
$\nu_k\downarrow 0$. 

Let us first suppose that $0<\bar\nu:=\inf_k\nu_k$. 
By lines (v) and (vi) of the algorithm, $0<\bar\eps:=\inf_k\eps_k$ and 
$\bar\nu\le\inf_{k\in J}\norm{g^k}$.  Therefore,  
\eqref{eq:finite sum} tells us that $\sum_{k=1}^\infty \norm{x^{k+1}-x^k}<\infty$
and $t_k\downarrow 0$. In particular, this implies that
the sequence $\{x^k\}$ is Cauchy, and so there
exists $\bx$ such that $x^k\rightarrow \bx$. Assume to the contrary that
none of the conditions in \eqref{eq:Limits} hold, or equivalently, the hypotheses
of Lemma \ref{lem:stepsize} \eqref{eq:limits} hold at $\bx$. 
Let $\eps,\,\eta,\, \del,\, \tau\in\bR_{++}$
and $\bV\subset\Rn$ be an open set satisfying the conditions of 
Lemma \ref{lem:stepsize}. We may assume that $\tau<\inf_k\eps_k$.
Since for all $k$ sufficiently large $t_k<\gam\eps/3$,
we must have $(x^{k1},\dots,x^{km})\notin \bV$ for all large $k$.
But since $\bV$ is open, the probability of this event is zero. 
Hence, with probability 1, Lemma \ref{lem:stepsize} tells us that at least
one of the three conditions in \eqref{eq:limits}
must hold, that is, (c) is satisfied.

Finally, suppose that $\nu_k\downarrow 0$. 
By line (vi) of the algorithm, 
$\eps_k\downarrow 0$. 
Let $\bx$ be a cluster point of the sequence $\{x^k\}$.
If there is any subsequence $\hJ\subset\bN$ such that
\begin{equation}\label{eq:good subseq}
x^k\overset{\hJ}{\rightarrow}\bx\ \text{ and }\ 
\norm{g^k}\overset{\hJ}{\rightarrow}0, 
\end{equation}
then 
$0\in\csd f(\bx)$ by Lemma \ref{lem:osc G}.
Therefore, we assume that no such subsequence exists
and establish a contradiction.
In particular, this implies that $x^k\not\rightarrow \bx$.
Since no subsequence satisfies \eqref{eq:good subseq}, 
there exist $\bnu>0$ such that if
$\norm{x^k-\bx}\le\bnu$, then $\norm{g^k}>\bnu$; otherwise, 
there exists $\hJ\subset\bN$ and $\bnu_k\downarrow_\hJ 0$ and  such that
$\norm{x^k-\bx}\le\bnu_k$ and $\norm{g^k}\le\bnu_k$ for all $k\in\hJ$
which implies that $\hJ$ satisfies
\eqref{eq:good subseq}, a contradiction.
Since $\bx$ is a cluster point, the set
$K:=\set{k}{\norm{x^k-\bx}\le \bar\nu}$ is infinite with $\norm{g^k}>\bar\nu$
for all $k\in K$. 
Since $x^k\not\rightarrow \bx$, we can reduce $\bar\nu$ if necessary so that the
set $\bN\setminus K$ is infinite. 
Observe that inequality \eqref{eq:finite sum} tell us that
$\sum_{k\in K}\norm{x^k-\bx}<\infty$. 
Let $\hK:=\set{k}{\norm{x^k-\bx}\le \bar\nu/3}$. Again $\hK$ is infinite
since $\bx$ is a cluster point of $\{x^k\}$.
Both $K$ and $\hK\subset K$ 
define subsequences of $\{x^k\}$. 
Since $\bN\setminus K$ is infinite,
for each $k\in \hK$ there is an
$\hk>k$ such that $\hk\notin K$ but $x^i\in K$ for $k\le i<\hk$. 
By construction, $\norm{x^\hk-x^k}\ge\bar\nu/3$
for all $k\in\hK$; otherwise, $x^\hk\in K$, a contradiction.
By the triangle inequality, we have
$\bar\nu/3\le\norm{x^\hk-x^k}\le\sum_{i=k}^{\hk -1}\norm{x^{i+1}-x^i}$
for all $k\in\hK$. But $\sum_{k\in K}\norm{x^k-\bx}<\infty$ 
and $\hK\subset K$ so that
$\sum_{i=k}^{\hk -1}\norm{x^{i+1}-x^i}\overset{\hK}{\rightarrow}0$.
This contradiction implies that our assumption that there is 
no subsequence satisfying \eqref{eq:good subseq} is false.
That is, $0\in\csd f(\bx)$.
\end{proof}

In the Lipschitzian case, Theorem \ref{thm:convergence} differs from
\cite[Theorem 3.3]{Kiwi07} with the introduction of possible outcome (c).
Kiwiel's proof of \cite[Theorem 3.3]{Kiwi07} shows that the case
$0<\bar\nu:=\inf_k\nu_k$ does not occur if $f$ is locally Lipschitz continuous.
The absence of the case (c) requires that $\csd f$ is an osc, compact, convex valued 
operator whose domain is all of $\Rn$, in particular, it requires that 
$f$ be locally Lipschitz. 
On the other hand, if $f$ is not locally Lipschitz, then 
$\csd f$ is not locally bounded and possibly empty at some points.
These possibilities are reflected in the outcome (c), and only in (c). 
This does not imply that $\csd f(\bx)$ is bounded in outcome (d), but 
outcome (d) does require that $\csd f(\bx)$ be nonempty. 
Note that outcome (c) signals why $\nu_k$ is not reduced to zero.
These observations are reviewed in our final comments. 
We conclude this section by stating two corollaries that describe the behavior of
the algorithm under standard variations in the choice of 
of initial parameters.

\begin{cor}[Convergence: $0<\eps_{opt},\ 0<\nu_{opt}$]
\label{cor:convergence1}
Suppose $\map{f}{\Rn}{\bR}$ satisfies $\cH$. 
Let $\{x^k\}$ be a sequence generated by the GS Algorithm with
$\nu_0,\, \eps_0\in\bR_{++}$, $\theta_\eps,\,\theta_\nu\in(0,1)$
and $0<\eps_{opt},\ 0<\nu_{opt}$.
With probability 1 the algorithm does not terminate 
in line \ref{step.terminate0} 
and 
one of the following must occur: 
\begin{enumerate}[label=(\alph*)]
\item 
There is a $k_0\in\bN$ such that 
$\dist{0}{\csd_{\eps_{opt}} f(x^{k_0})}\le\nu_{opt}$ and 
the algorithm terminates.
\item 
$f(x^k)\downarrow -\infty$. 
\item \label{limit2.1}
$\nu_{opt}<\bar\nu:=\inf_k\nu_k$ and the sequence converges 
to some $\bx\in\Rn$ at which at least one of the three
statements in \eqref{eq:Limits} is true. 
\end{enumerate}
\end{cor}

\begin{proof}
By assumption the algorithm does not terminate in line \ref{step.terminate0}
of the GS Algorithm
with probability 1, so we assume 
$\{x^{k,1},\dots,x^{k,m}\}\subset\cD$ for all $k$.
Next we assume that neither (a) nor (b) occur and show that (c) must occur.
Since (a) does not occur and $\nabla f(x^k)\in\csd_{\eps_k}f(x^k)$, the algorithm does not terminate in step (iv)
and step (v) of the algorithm occurs at most finitely many times. 
Therefore,
$\nu_{opt}<\bar\nu:=\inf_k\nu_k$, 
the algorithm does not terminate and the sequence $\{x^k\}$ is infinite.
Consequently, Theorem \ref{thm:convergence} tells us that (c) must occur
and the final statement of the corollary follows.
\end{proof}

\begin{cor}[Convergence: $0\!<\!\eps_{opt}\!=\!\eps_0,\, 0\!=\!\nu_{opt}\!=\!\nu_0$]
\label{cor:convergence2}
Suppose $\map{f}{\Rn}{\bR}$ satisfies $\cH$. 
Let $\{x^k\}$ be a sequence generated by the GS Algorithm with
$\nu_{opt}=\nu_0=0$, 
$\eps_{opt}=\eps_0>0$ and $0=\theta_\nu,\ 1=\theta_\eps$.
Let $J\subset\bN$ be those iterations for which $x^k\ne x^{k+1}$.
With probability 1 the algorithm does not terminate 
in line \ref{step.terminate0} 
and 
one of the following must occur: 
\begin{enumerate}[label=(\alph*)]
\item 
The algorithm terminates at some iteration $k_0\in\bN$ with either 
$\nabla f(x^{k_0})=0$ or $g^{k_0}=0$, and consequently
 $0\in\csd_{\eps_{opt}} f(x^{k_0})$.
\item 
$f(x^k)\downarrow -\infty$. 
\item \label{limit2.1c}
The sequence $\{x^k\}$ is infinite with $\inf_{k\in J}\norm{g^k}>0$ in which case
there exists $\bx\in\Rn$ such that
$x^k\rightarrow \bx$ and at least one of the conditions in 
\eqref{eq:limits} is satisfied. 
\item \label{limit2.2c}
The sequence $\{x^k\}$ is infinite with $\inf_{k\in J}\norm{g^k}=0$ in which case
every cluster point $\bx$ of $\{x^k\}$
(if one exists) 
satisfies $0\in\csd f(\bx)$.
\end{enumerate}
\end{cor}
\begin{proof}
The proof strategy follows that of the Theorem \ref{thm:convergence}.
By assumption the algorithm does not terminate in line \ref{step.terminate0}
of the GS Algorithm
with probability 1, so we assume 
$\{x^{k,1},\dots,x^{k,m}\}\subset\cD$ for all $k$.
We also assume that neither (a) nor (b) occur and show that either 
(c) or (d) must occur. 
Observe that if $J$ is finite with
maximum value $k_0$, then, by step (iv) of the algorithm, (a) occurs, 
hence, $J$ is infinite. Following the proof of Theorem \ref{thm:convergence}, 
we have that 
\eqref{eq:finite sum} holds.
We analyze the two mutually exclusive possible 
outcomes $\inf_{k\in J}\norm{g^k}>0$ and $\inf_{k\in J}\norm{g^k}=0$ separately.

First suppose that $\bnu:=\inf_{k\in J}\norm{g^k}>0$. By \eqref{eq:finite sum}, the
sequence $\{x^k\}$ is Cauchy so that $x^k\rightarrow \bx$ for some 
$\bx\in\Rn$. The argument used in Theorem \ref{thm:convergence} 
applies to show that
one of the conditions in \eqref{eq:limits} is satisfied.

Next suppose that $\inf_{k\in J}\norm{g^k}=0$ and $\bx$ is
a cluster point of the sequence $\{x^k\}$.
As in the proof
of Theorem \ref{thm:convergence}, assume that there is no subsequence
$\hJ\subset\bN$ satisfying \eqref{eq:good subseq}.  Following the proof of
Theorem \ref{thm:convergence}, we again find that $0\in\csd f(\bx)$.
\end{proof}

\section{Concluding Remarks}

The extension of the gradient sampling algorithm to 
non-Lipschitzian, continuous, directionally Lipschitz functions
addresses the possibility of unbounded and potentially
empty Clarke subdifferentials. These possibilities effect both
the construction of the algorithm and the convergence
results. Specifically, in line (vii) of the algorithm, we require that
the direction of steepest descent be normalized to have unit magnitude
since it may happen that the sequence $\{g^k\}$ is unbounded.
Although other normalization strategies are possible, we chose a unit normalization for simplicity. 
As for the convergence results, the results differ from the Lipschitzian case only by the inclusion of outcome (c) in Theorem 
\ref{thm:convergence} as well as Corollaries \ref{cor:convergence1}
and \ref{cor:convergence2}. This outcome occurs only if
the sequence $\{g^k\}$ does not converge to zero in which case
it is shown that the sequence $\{x^k\}$ converges to a limit $\bx$. Lemma 
\ref{lem:stepsize} indicates that this can be manifested in
excessively short stepsizes. Nonetheless, in this case
failure to converge
to a Clarke stationary point only occurs when either 
$\csd f( \bx)=\emptyset$ or
$\csd f( \bx)$ is unbounded and 
\[
-\proj{\csd f(\bx)}(0)\notin\intr[(\hcsd f(\bx))^*]
=\intr{\left[ \dom{\hat d f(\bx)(\cdot)}\right]},
\] 
or equivalently, the regular subderivative $\hat d f(\bx)(\cdot)$ is not
continuous at the direction of steepest descent 
(see Lemma \ref{lem:interior equiv}).
This observation yields two open question in the directionally
Lipschitz case. First, is it possible for $\csd f( \bx)=\emptyset$, and if so, when does this occur? 
Second, is there a away to modify the search direction so that
the iterates are not attracted to non-stationary points at
which  $-\proj{\csd f(\bx)}(0)\notin \intr (\hcsd f(\bx))^*$, or is this
fundamental limitation of the method?

Finally, we note that the class of directionally Lipschitz functions
is still not sufficiently broad to capture the non-symmetric spectral functions
even though the method has successfully been applied in
this case
\cite{BurkLewiOver02b,BurkLewiOver03,BurkLewiOver05}. 
For these functions, there is still much more work to do and it is likely
that a very different approach 
to the convergence analysis is required.

\end{document}